\documentclass[reqno,12pt,a4paper]{amsart}

\usepackage{tikz}
\usetikzlibrary{matrix,arrows,calc,snakes,patterns,decorations.markings}

\usepackage[mathscr]{eucal}
\usepackage{graphics,epic}
\usepackage{amsfonts}
\usepackage{amscd}
\usepackage{latexsym}
\usepackage{amsmath,amssymb, amsthm, stmaryrd, bm, bbm}
\usepackage[all,2cell]{xy}
\usepackage{mathrsfs}
\usepackage{color}

\newtheorem{theorem}[subsection]{Theorem}
\newtheorem*{theorem*}{Theorem}

\newtheorem{lemma}[subsection]{Lemma}
\newtheorem{proposition}[subsection]{Proposition}
\newtheorem{corollary}[subsection]{Corollary}

\newtheorem*{conjecture*}{Conjecture}

\newtheorem*{question*}{Question}
\theoremstyle{remark}

\newtheorem{example}[subsection]{Example}
\theoremstyle{definition}
\newtheorem{definition}[subsection]{Definition}
\newtheorem*{notation*}{Notation}

\newcommand{\opname}[1]{\operatorname{\mathsf{#1}}}

\renewcommand{\mod}{\opname{mod}\nolimits}

\newcommand{\proj}{\opname{proj}\nolimits}
\newcommand{\inj}{\opname{inj}\nolimits}

\newcommand{\add}{\opname{add}\nolimits}

\newcommand{\Hom}{\opname{Hom}}
\newcommand{\End}{\opname{End}}
\newcommand{\Aut}{\opname{Aut}}

\newcommand{\Ext}{\opname{Ext}}

\newcommand{\silt}{\opname{silt}}
\newcommand{\tilt}{\opname{tilt}}
\newcommand{\tors}{\opname{tors}}

\numberwithin{equation}{section}

\begin{document}

\title[]{$G$-stable support $\tau$-tilting modules}

\author{Yingying Zhang and Zhaoyong Huang}

\address{Department of Mathematics, Nanjing University, Nanjing 210093, Jiangsu Province, P.R. China}

\email{zhangying1221@sina.cn}
\address{Department of Mathematics, Nanjing University, Nanjing 210093, Jiangsu Province, P.R. China}

\email{huangzy@nju.edu.cn}

\begin{abstract}
Motivated by $\tau$-tilting theory developed by Adachi, Iyama and Reiten, for a finite-dimensional algebra $\Lambda$ with action by
a finite group $G$, we introduce the notion of $G$-stable support $\tau$-tilting modules.
Then we establish bijections among $G$-stable support $\tau$-tilting modules over $\Lambda$, $G$-stable two-term
silting complexes in the homotopy category of bounded complexes of finitely generated projective $\Lambda$-modules,
and $G$-stable functorially finite torsion classes in the category of finitely generated left $\Lambda$-modules.
In the case when $\Lambda$ is the endomorphism of a $G$-stable cluster-tilting object $T$ over a $\Hom$-finite
2-Calabi-Yau triangulated category $\mathcal{C}$ with a $G$-action, these are also in bijection with $G$-stable
cluster-tilting objects in $\mathcal{C}$. Moreover, we investigate the relationship between stable support $\tau$-tilitng modules
over $\Lambda$ and the skew group algebra $\Lambda G$.
\end{abstract}

\maketitle

\section{Introduction}\label{s:introduction}
\medskip
It is well known that tilting theory is a theoretical basis in the representation theory of finite-dimensional algebras,
in which the notion of tilting modules is fundamental. Moreover, in the representation theory of algebras the notion of ``mutation"
often plays an important role. Mutation is an operation for a certain class of objects in a fixed category to construct a new object
from a given one by replacing a summand. In [HU2], Happel and Unger gave some necessary and sufficient conditions under which mutation
of tilting modules is possible; however, mutation of tilting modules is not always possible. In [AIR], Adachi, Iyama and Reiten
introduced the notion of support $\tau$-tilting modules which generalizes that of tilting modules, and showed that
mutation of support $\tau$-tilting modules is always possible. This is a big advantage of ``support $\tau$-tilting mutation"
which ``tilting mutation" does not have.
Note that the $\tau$-tilting theory developed in [AIR] has stimulated several investigations;
in particular, there is a close relation between support $\tau$-tilting modules and some other
important notions in the representation theory of algebras, such as torsion classes, silting complexes, cluster-tilting objects,
Grothendieck groups and $*$-modules, see [AIR], [AiI], [BDP], [IJY], [IR], [J], [W], and so on.
Moreover, Adachi gave in [A] a classification of $\tau$-tilting modules over Nakayama algebras
and an algorithm to construct the exchange quiver of support $\tau$-tilting modules.
Zhang studied in [Z] $\tau$-rigid modules which are direct summands of support $\tau$-tilting modules
over algebras with radical square zero.

On the other hand, the notion of skew group algebras was introduced in [RR]. Let $\Lambda$ be a finite-dimensional algebra and $G$ a
finite group such that its order $|G|$ is invertible in $\Lambda$ acting on $\Lambda$. The algebra $\Lambda$ and the skew group algebra
$\Lambda G$ have a lot of properties in common.

The aim of this paper is to introduce and study $G$-stable support $\tau$-tilting modules and moreover to establish bijections among them
and $G$-stable two-term silting complexes, $G$-stable functorially finite torsion classes, and $G$-stable cluster-tilting objects. Moreover, we investigate the relationship between stable support $\tau$-tilting modules over $\Lambda$ and the skew group algebra $\Lambda G$.
This paper is organized as follows.

In Section 2, we give some terminology and some known results.

In Section 3, we prove the following theorem.

\begin{theorem}
{\rm ([Theorems 3.4, 3.7, 3.13 and 3.15])}\label{t:integer2}
Let $\Lambda$ be a finite-dimensional algebra and $G$ a finite group acting on $\Lambda$. Then there exist bijections among
\begin{itemize}
\item[(1)] the set $G$-$s\tau$-$\tilt\Lambda$ of isomorphism classes of basic $G$-stable support $\tau$-tilting modules in $\mod\Lambda$;
\item[(2)] the set $G$-$2$-$\silt\Lambda$ of isomorphism classes of basic $G$-stable two-term silting complexes in $\opname{K}^{b}(\proj \Lambda)$;
\item[(3)] the set $G$-$f$-$\tors\Lambda$ of $G$-stable functorially finite torsion classes in $\mod\Lambda$.
\end{itemize}

Furthermore, if $\Lambda=\opname{End}_{\mathcal{C}}(T)$ where $\mathcal{C}$ is a $\Hom$-finite 2-Calabi-Yau triangulated category with
a $G$-action and $T$ is a $G$-stable cluster-tilting object, then there exists also a bijection between the following set and
any one of the above sets.
\begin{itemize}
\item[(4)] the set $G$-$c$-$\tilt\mathcal{C}$ of isomorphism classes of basic $G$-stable cluster-tilting objects in $\mathcal{C}$.
\end{itemize}
\end{theorem}

Let $\Lambda$ be a finite-dimensional self-injective algebra and $G=\langle\nu\rangle$ the subgroup of
the automorphism group of $\Lambda$ generated by the Nakayama automorphism.
Then $\nu$-stable support $\tau$-tilting modules introduced by Mizuno in [M] are exactly $G$-stable support $\tau$-tilting modules in our sense.
So Theorem 1.1 is a generalization of [M, Theorem 1.1].

In Section 4, we investigate the relationship between $G$-stable support $\tau$-tilting $\Lambda$-modules
and $\mathbb{X}$-stable support $\tau$-tilting $\Lambda G$-modules,
where $\mathbb{X}$, the group of characters of $G$, naturally acts on $\Lambda G$. We have the following

\begin{theorem}
{\rm ([Theorems 4.2(3) and 4.6])}\label{t:integer2}
Let $\Lambda$ be a finite-dimensional algebra and $G$ a finite group acting on $\Lambda$ such that $|G|$ is invertible in $\Lambda$.
Then the functor $\Lambda G\otimes_{\Lambda}-: \mod\Lambda \rightarrow \mod\Lambda G$ preserves stability and induces the following injection:
\begin{center}
$G$-$s\tau$-$\tilt\Lambda \rightarrow \mathbb{X}$-$s\tau$-$\tilt\Lambda G$ via $T \mapsto \Lambda G\otimes_{\Lambda}T$.
\end{center}
Moreover, if $G$ is solvable, then this map is a bijection.
\end{theorem}

Finally we give an example to illustrate this theorem.

\section{Preliminaries}

\vspace{0.2cm}

In this section, we give some terminology and some known results.

Let $k$ be an algebraically closed field and we denote by $D:=\Hom_{k}(-, k)$. By an algebra $\Lambda$,
we mean a finite-dimensional algebra over $k$. We denote by $\mod\Lambda$ the category of finitely
generated left $\Lambda$-modules, by $\proj \Lambda$ and $\inj \Lambda$ the subcategories of $\mod \Lambda$
consisting of projective modules and injective modules respectively,
and by $\tau$ the Auslander-Reiten translation of $\Lambda$. We denote by $\opname{K}^{b}(\opname{proj\Lambda})$
the homotopy category of bounded complexes of $\proj \Lambda$. For $X\in\mod\Lambda$,
we denote by $\add X$ the subcategory of $\mod\Lambda$ consisting of all direct summands of finite direct sums of copies of $X$,
and by $\opname{Fac}X$ the subcategory of $\mod \Lambda$ consisting of all factor modules of finite direct sums of copies of $X$.

\vspace{0.2cm}

{\bf 2.1. $\tau$-tilting theory}

\vspace{0.2cm}

First we recall the definition of support $\tau$-tilting modules from [AIR].

\vspace{0.2cm}

Let $(X,P)$ be a pair with $X \in \mod \Lambda$ and $P \in \proj \Lambda$.
\begin{itemize}
\item[(1)] We call $X$ in $\mod\Lambda$ {\it $\tau$-rigid} if  $\Hom_{\Lambda}(X, \tau X)=0$.
We call $(X,P)$ a {\it $\tau$-rigid pair} if $X$ is $\tau$-rigid and $\Hom_{\Lambda}(P, X)=0$.
\item[(2)] We call $X$ in $\mod\Lambda$ {\it $\tau$-tilting} (respectively, {\it almost complete $\tau$-tilting})
if $X$ is $\tau$-rigid and $|X|=|\Lambda|$ (respectively, $|X|=|\Lambda|-1$), where $|X|$ denotes the number
of non-isomorphic indecomposable direct summands of $X$.
\item[(3)] We call $X$ in $\mod\Lambda$ {\it support $\tau$-tilting} if there exists an idempotent $e$ of
$\Lambda$ such that $X$ is a $\tau$-tilting ($\Lambda/\langle e\rangle$)-module. We call $(X,P)$ a
{\it support $\tau$-tilting pair} (respectively, {\it almost complete support $\tau$-tilting pair})
if $(X,P)$ is $\tau$-rigid and $|X|+|P|=|\Lambda|$ (respectively, $|X|+|P|=|\Lambda|-1$).
\end{itemize}

We say that $(X,P)$ is {\it basic} if $X$ and $P$ are basic. Moreover, $X$ determines $P$ uniquely up to isomorphism.
We denote by $s$$\tau$-$\tilt\Lambda$ the set of isomorphism classes of basic support $\tau$-tilting $\Lambda$-modules.


\begin {proposition} {\rm ([AIR, Proposition 2.3])}
Let $X \in \mod \Lambda$ and $P, Q \in \proj  \Lambda$, and let $e$ be an idempotent of $\Lambda$ such that $\add P=\add \Lambda e$.
\begin{itemize}
\item[(1)] $(X,P)$ is a $\tau$-rigid pair for $\Lambda$ if and only if $X$ is a $\tau$-rigid $(\Lambda/\langle e\rangle)$-module.
\item[(2)] If both $(X,P)$ and $(X,Q)$ are support $\tau$-tilting pairs for $\Lambda$, then $\add P=\add Q$. In other words,
$X$ determines $P$ and e uniquely up to equivalence.
\end{itemize}
\end{proposition}

Let $\mathcal{T}$ be a full subcategory of $\mod\Lambda$.
Assume that $T\in \mathcal{T}$ and $D\in\mod \Lambda$. The morphism $f: D\to T$ is called a
{\it left $\mathcal{T}$-approximation} of $D$ if
$$\Hom_{\Lambda}(T, T')\to \Hom_{\Lambda}(D, T')\to 0$$
is exact for any $T'\in\mathcal{T}$.
The subcategory $\mathcal{T}$ is called {\it covariantly finite} in $\mod \Lambda$ if
every module in $\mod \Lambda$ has a left $\mathcal{T}$-approximation. The notions of {\it right
$\mathcal{T}$-approximations} and {\it contravariantly finite subcategories} of $\mod \Lambda$ are defined dually. The
subcategory $\mathcal{T}$ is called {\it functorially finite} in $\mod \Lambda$ if it is both covariantly finite and contravariantly
finite in $\mod \Lambda$ ([AR]).

Recall that $T \in \mod \Lambda$ is called {\it partial tilting} if the projective dimension of $T$ is at most one and
$\Ext^1_{\Lambda}(T, T)=0$. A partial tilting module is called {\it tilting} if there exists an exact sequence
$$0\rightarrow \Lambda\rightarrow T'\rightarrow T''\rightarrow 0$$ in $\mod \Lambda$ with $T'$, $T''\in\add T$ (see [HU1] and [B]).
We have that $|T|=|\Lambda|$ for any tilting module $T$ by [B, Theorem 2.1].
The following result gives a similar criterion for a $\tau$-rigid $\Lambda$-module to
be support $\tau$-tilting.

\begin{proposition} {\rm ([J, Proposition 2.14])} Let $M$ be a $\tau$-rigid $\Lambda$-module. Then $M$ is a support
$\tau$-tilting $\Lambda$-module if and only if there exists an exact sequence
$$\Lambda\buildrel {f}\over\rightarrow M'\buildrel {g} \over\rightarrow M''\rightarrow 0$$ in $\mod \Lambda$
with $M', M''\in \add M$ and $f$ a left $\add M$-approximation of $\Lambda$.
\end{proposition}

\vspace{0.2cm}

{\bf 2.2. Functorially finite torsion classes}

\vspace{0.2cm}

Let $\mathcal{T}$ be a full subcategory of $\mod\Lambda$. Recall that $\mathcal{T}$ is called a {\it torsion class} if it is closed under
factor modules and extensions. We denote by $f$-$\tors\Lambda$ the set of functorially finite torsion classes in $\mod\Lambda$.
We say that $X \in \mathcal{T}$ is $\Ext$-{\it projective} if $\Ext^{1}_{\Lambda}(X,\mathcal{T})$=0. We denote by $P(\mathcal{T}$)
the direct sum of one copy of each of the indecomposable $\Ext$-projective objects in $\mathcal{T}$ up to isomorphism. We have that
$P(\mathcal{T})\in\mod\Lambda$ if $\mathcal{T}\in f$-$\tors\Lambda$ ([AS, Corollary 4.4]).
The following result establishes a relation between $s\tau$-$\tilt\Lambda$ and $f$-$\tors\Lambda$.

\begin{theorem} {\rm ([AIR, Theorem 2.7])}
There exists a bijection:
\begin{center}
$s\tau$-$\tilt\Lambda \longleftrightarrow f$-$\tors\Lambda$
\end{center}
given by $s\tau$-$\tilt\Lambda \ni T \mapsto\opname{Fac}T \in f$-$\tors\Lambda$ and
$f$-$\tors\Lambda \ni \mathcal{T} \mapsto P(\mathcal{T}) \in s\tau$-$\tilt\Lambda$.
\end{theorem}

\vspace{0.2cm}

{\bf 2.3. Silting complexes}

\vspace{0.2cm}


Recall from [AiI] that $P \in \opname K^{b}(\proj \Lambda)$ is called {\it silting} if
$\Hom_{\opname K^{b}(\proj \Lambda)}(P ,P[i])=0$ for any $i>0$ and $\opname K^{b}(\proj \Lambda)$ is the smallest
full subcategory of $\opname K^{b}(\proj \Lambda)$ containing $P$ and is closed under cones, $[\pm 1]$ and direct summands;
and a complex $P=(P^{i},d^{i})$ in $\opname K^{b}(\proj \Lambda)$ is called {\it two-term} if $P^{i}=0$ for all $i\neq0, -1$.
We denote by $2$-$\silt\Lambda$ the set of isomorphism classes of basic two-term silting complexes in $\opname K^{b}(\proj \Lambda)$.
The following result establishes a relation between $2$-$\silt\Lambda$ and $s\tau$-$\tilt\Lambda$.

\begin{theorem} {\rm ([AIR, Theorem 3.2])}
There exists a bijection:
\begin{center}
$2$-$\silt\Lambda\longleftrightarrow s\tau$-$\tilt\Lambda$
\end{center}
given by $2$-$\silt\Lambda \ni P\mapsto H^{0}(P) \in s\tau$-$\tilt\Lambda$ and $s\tau$-$\tilt\Lambda \ni
(M,P)\mapsto (P_{1}\oplus P \buildrel {(f, 0)} \over \rightarrow P_{0}) \in 2$-$\silt\Lambda$, where $f:P_{1}\rightarrow P_{0}$
is a minimal projective presentation of $M$.
\end{theorem}

\vspace{0.2cm}

{\bf 2.4. Cluster tilting objects}

\vspace{0.2cm}

Let $\mathcal{C}$ be a $k$-linear $\Hom$-finite Krull--Schmidt triangulated category. Assume that $\mathcal{C}$ is a
2-Calabi-Yau triangulated category, that is, there exists a functorial isomorphism:
$$D \Ext^{1}_{\mathcal{C}}(X,Y)\cong \Ext^{1}_{\mathcal{C}}(Y,X).$$ An important class of objects in such categories is
that of cluster-tilting objects. Following [BMRRT], an object $T\in\mathcal{C}$ is called {\it cluster-tilting} if
\begin{center}
$\add T=\{X\in \mathcal{C}\mid \Hom_{\mathcal{C}}(T,X[1])=0\}$.
\end{center}
We denote by $c$-$\tilt\mathcal{C}$ the set of isomorphism classes of basic cluster-tilting objects in $\mathcal{C}$.
Assume that $\mathcal{C}$ has a cluster-tilting object $T$ and $\Lambda:=\End_{\mathcal{C}}(T)^{op}$. For $X\in \mathcal{C}$, we have a triangle
$$T_{1}\buildrel {g} \over\rightarrow T_{0}\buildrel {f} \over\rightarrow X\rightarrow T_{1}[1], \eqno{(\ast)}$$
where $T_{1}, T_{0}\in \add T$ and $f$ is a minimal right $\add T$-approximation.

We have the following results, which will be used frequently in this paper.
\begin{theorem} {\rm ([BMR, Theorem 2.2] and [KR, p.126])}
There exists an equivalence of categories
\begin{center}
$\overline{(-)}:=\Hom_{\mathcal{C}}(T,-):\mathcal{C}/[T[1]]\rightarrow \mod\Lambda$,
\end{center}
where $[T[1]]$ is the ideal of $\mathcal{C}$ consisting of morphisms which factor through $\add T[1]$.
\end{theorem}

\begin{theorem} {\rm ([AIR, Theorem 4.1])}
There exists a bijection:
\begin{center}
$c$-$\tilt\mathcal{C}\longleftrightarrow s\tau$-$\tilt\Lambda$
\end{center}
given by $c$-$\tilt\mathcal{C}\ni X=X'\oplus X''\mapsto \widetilde{X}:=(\overline{X'},
\overline{X''[-1])})\in s\tau$-$\tilt\Lambda$, where $X''$ is a maximal direct summand of $X$ belonging to $\add T[1]$.
\end{theorem}

\begin{theorem} {\rm ([AIR, Theorem 4.7])}
There exists a bijection:
\begin{center}
$c$-$\tilt\mathcal{C}\longleftrightarrow $2$ $-$\silt\Lambda$
\end{center}
given by $c$-$\tilt\mathcal{C}\ni X\mapsto (\overline{T_{1}}\buildrel {\overline{g}}
\over\rightarrow\overline{T_{0}})\in $2$ $-$\silt\Lambda$, where g is the morphism in $(\ast)$.
\end{theorem}

\vspace{0.2cm}

{\bf 2.5. Skew group algebras}

\vspace{0.2cm}

In this subsection, we recall the definition of skew group algebras and some useful results from [RR].

Let $\Lambda$ be an algebra and $G$ be a group with identity 1. Consider an action of $G$ on $\Lambda$,
that is a map $G\times \Lambda \rightarrow \Lambda$ via $(\sigma,\lambda)\mapsto \sigma(\lambda)$ such that

\begin{itemize}
\item[(1)] For any $\sigma$ in $G$, the map $\sigma: \Lambda \rightarrow \Lambda$ is an algebra automorphism.
\item[(2)] $(\sigma \sigma')(\lambda)=\sigma(\sigma'(\lambda))$ for any $\sigma, \sigma'\in G$ and $\lambda\in\Lambda$.
\item[(3)] 1$(\lambda)=\lambda$ for any $\lambda \in\Lambda$.
\end{itemize}

Let $G$ be a finite group.
For any $X\in \mod\Lambda$ and $\sigma\in G$,
let $^{\sigma}\hspace{-2pt}X$ be a $\Lambda$-module as follows: as a $k$-vector space $^{\sigma}\hspace{-2pt}X=X$,
the action on $^{\sigma}\hspace{-2pt}X$ is given by $\lambda\cdot x=\sigma^{-1}(\lambda)x$ for any $\lambda\in\Lambda$
and $x\in X$. Given a morphism of $\Lambda$-modules $f:X\rightarrow Y$, define
$^{\sigma}\hspace{-2pt}f: {^{\sigma}\hspace{-2pt}X}$ $\rightarrow$ $^{\sigma}\hspace{-1pt}Y$
by $^{\sigma}\hspace{-2pt}f(x)=f(x)$ for any $x\in$ $^{\sigma}\hspace{-2pt}X$. Then $^{\sigma}\hspace{-2pt}f$
is also a $\Lambda$-homomorphism. Indeed, for any $x\in X$ and $\lambda\in\Lambda$, we have
$$^{\sigma}\hspace{-2pt}f(\lambda\cdot x)=f(\sigma^{-1}(\lambda)x)=\sigma^{-1}(\lambda)f(x)=\lambda\cdot{^{\sigma}\hspace{-2pt}f(x)}.$$
For any $X,Y\in \mod\Lambda$ and
$f\in \Hom_{\Lambda}(X,Y)$, we define a functor $^{\sigma}(-)$  by  $^{\sigma}(-)(X)$=$^{\sigma}\hspace{-2pt}X$ and
$^{\sigma}(-)(f)={^{\sigma}\hspace{-2pt}f}$. One can check that $^{\sigma}(-):\mod\Lambda\rightarrow \mod\Lambda$
is an automorphism and the inverse is $^{\sigma ^{-1}}(-)$.
So we have that $X\in \mod\Lambda$ is indecomposable (respectively, projective, injective, simple) if and only if so is
$^{\sigma}\hspace{-2pt}X$ in $\mod\Lambda$.

The {\it skew group algebra} $\Lambda G$ of $G$ over $\Lambda$ is given by the following data:

\begin{itemize}
\item[(1)] As an abelian group $\Lambda G$ is a free left $\Lambda$-module with the elements of $G$ as a basis.
\item[(2)] The multiplication in $\Lambda G$ is defined by the rule $(\lambda_{\sigma}\sigma)
(\lambda_{\tau}\tau)=(\lambda_{\sigma}\sigma(\lambda_{\tau}))\sigma\tau$ for any
$\lambda_{\sigma},\lambda_{\tau}\in\Lambda$ and $\sigma,\tau\in G$.
\end{itemize}

When $G$ is a finite group such that $|G|$ is invertible in $\Lambda$, the natural inclusion $\Lambda\hookrightarrow\Lambda G$
induces the induction functor $F=\Lambda G\otimes _{\Lambda}-: \mod\Lambda\rightarrow \mod\Lambda G$
and the restriction functor $H: \mod\Lambda G\rightarrow \mod\Lambda$.

\begin{lemma} {\rm ([RR, p.227 and p.235])} Let $G$ be a finite group such that $|G|$ is invertible in $\Lambda$. Then we have
\begin{itemize}
\item[(1)] $(F,H)$ and $(H,F)$ are adjoint pairs of functors. Consequently, $F$ and $H$ are both exact,
and hence both preserve projective modules and injective modules.
\item[(2)] Let $M\in \mod\Lambda$ and $\sigma\in G$. Then the subset $\sigma\otimes_{\Lambda}M=
\{\sigma\otimes_{\Lambda}m | m\in M\}$ of $FM$ has a structure of $\Lambda$-module given by
$$\lambda(\sigma\otimes_{\Lambda}m)=\sigma\sigma^{-1}(\lambda)\otimes_{\Lambda}m=\sigma\otimes_{\Lambda}(\lambda\cdot m)$$
for any $\lambda\in\Lambda$, so that $\sigma\otimes_{\Lambda}M$
and $^{\sigma}\hspace{-2pt}M$ are isomorphic as $\Lambda$-modules. Therefore, as $\Lambda$-modules, we have
$$FM\cong \bigoplus_{\sigma\in G}(\sigma\otimes_{\Lambda} M)\cong\bigoplus_{\sigma\in G}{^{\sigma}\hspace{-2pt}M},$$
and then $$HFM\cong \bigoplus_{\sigma\in G}(\sigma\otimes_{\Lambda} M)\cong\bigoplus_{\sigma\in G}{^{\sigma}\hspace{-2pt}M}.$$
\end{itemize}
\end{lemma}

\section{$G$-stable support $\tau$-tilting modules}

\vspace{0.2cm}

In this section we investigate the relationship among $G$-stable support $\tau$-tilting modules,
$G$-stable two-term silting complexes, $G$-stable functorially finite torsion classes and
$G$-stable cluster-tilting objects. From now on, $\Lambda$ is an algebra with action by a finite group $G$.

\vspace{0.2cm}

{\bf 3.1. Some definitions}

\vspace{0.2cm}

In this subsection, we introduce the notions of $G$-stable support $\tau$-tilting modules, $G$-stable torsion classes
and $G$-stable two-term silting complexes.

Recall from [DLS] that a tilting $\Lambda$-module $T$
is called {\it $G$-stable} if $^{\sigma} T\cong T$ for any $\sigma \in G$. Motivated by this, we introduce
the following

\begin{definition}

\begin{itemize}
\item[]
\item[(1)] We say that a support $\tau$-tilting module $X$ in $\mod\Lambda$ is {\it $G$-stable}
if $^{\sigma}\hspace{-2pt} X\cong X$ for any $\sigma \in G$.
\item[(2)] We say that a support $\tau$-tilting pair (or a $\tau$-rigid pair) $(X, P)$ for $\Lambda$ is {\it $G$-stable}
if $^{\sigma}\hspace{-2pt} X\cong X$ and $^{\sigma}\hspace{-2pt} P\cong P$ for any $\sigma \in G$.
\item[(3)] We say that a torsion class $\mathcal{T}$ is {\it $G$-stable} if $^{\sigma} \mathcal{T}=\mathcal{T}$ for any $\sigma \in G$.
\end{itemize}
\end{definition}

We denote by $G$-$s\tau$-$\tilt\Lambda$ the set of isomorphism classes of basic $G$-stable support $\tau$-tilting $\Lambda$-modules
and $G$-$f$-$\tors\Lambda$ the set of $G$-stable functorially finite torsion classes in $\mod\Lambda$.

The following result shows that in a support $\tau$-tilting pair $(T, P)$ the $G$-stability of $T$ implies the $G$-stability of the pair.

\begin{proposition} Let $T$ be a $\Lambda$-module and $P$ a projective $\Lambda$-module.
Then $(T, P)\in s\tau$-$\tilt\Lambda$ if and only if $(^{\sigma}T, {^{\sigma}P})\hspace{-2pt}\in s\tau$-$\tilt\Lambda$.
Moreover, if $(T, P)\in s\tau$-$\tilt\Lambda$ and $T$ is a $G$-stable support $\tau$-tilting module, then $P$ is $G$-stable.
\end{proposition}

\begin{proof} Since $^{\sigma}(-)$ is an automorphism commuting with $\tau$ (see the proof of [RR, Lemma 4.1]), we have that
$|T|+|P|=|\Lambda|$ if and only if $|^{\sigma}\hspace{-1pt}T|+|^{\sigma}\hspace{-2pt}P|=|\Lambda|$, and that
$\Hom_{\Lambda}(T, \tau T)=0$ if and only if $\Hom_{\Lambda}(^{\sigma}T,{^{\sigma}\tau T})=0$, and if and only if
$\Hom_{\Lambda}(^{\sigma}T, \tau{^{\sigma}T})=0$. So $T$ is $\tau$-rigid if and only if $^{\sigma}\hspace{-1pt}T$ is $\tau$-rigid.
Thus the former assertion follows.
If $T$ is a $G$-stable support
$\tau$-tilting module, then by Proposition 2.1 we have that $P$ is also $G$-stable.
\end{proof}

For any complex $M^{\bullet}$=$( M^{i}, d^{i}_{M^{\bullet}})$$_{i\in \mathbb{Z}}$ over $\mod\Lambda$ and $\sigma\in G$,
let $^{\sigma}\hspace{-2pt}M^{\bullet}$ be the complex $(^{\sigma}\hspace{-2pt}M^{i},
{^{\sigma}\hspace{-3pt}d^{i}_{M^{\bullet}}})_{i\in \mathbb{Z}}$, where $\mathbb{Z}$ is the ring of integers. Moreover, given another complex
$N^{\bullet}=(N^{i}, d^{i}_{N^{\bullet}})_{i\in \mathbb{Z}}$ over $\mod\Lambda$ and a morphism of complexes
$f=(f^{i}: M^{i}\rightarrow N^{i})_{i\in \mathbb{Z}}$, let $^{\sigma}\hspace{-2pt}f=(^{\sigma}\hspace{-2pt}f^{i}:
{^{\sigma}\hspace{-2pt}M^{i}} \rightarrow {^{\sigma}\hspace{-2pt}N^{i}})$$_{i\in \mathbb{Z}}$. Clearly,
$^{\sigma}\hspace{-2pt}f$ is a morphism of complexes.

Since $^{\sigma}(-): \mod\Lambda \rightarrow \mod\Lambda$ is an automorphism, this construction is compatible
with the homotopy relation and preserves projective modules. This allows defining an automorphism $^{\sigma}(-):
\opname K^{b}(\proj \Lambda) \rightarrow \opname K^{b}(\proj \Lambda)$ for any $\sigma\in G$. In this way, we obtain an action by
$G$ on $\opname K^{b}(\proj \Lambda)$.

\begin{definition}
We call a basic two-term silting complex $P^{\bullet}\in \opname K^{b}(\proj \Lambda)$ {\it $G$-stable}
if $^{\sigma}\hspace{-2pt}P^{\bullet}\cong P^{\bullet}$ for any $\sigma\in G$.
\end{definition}

We denote by $G$-$2$-$\silt\Lambda$ the set of isomorphism classes of basic $G$-stable two-term silting complexes for $\Lambda$.

\vspace{0.2cm}

{\bf 3.2. Connection of $G$-$s\tau$-$\tilt\Lambda$ with $G$-$f$-$\tors\Lambda$ and $G$-$2$-$\silt\Lambda$}

\vspace{0.2cm}

In this subsection, we show that $G$-stable support $\tau$-tilting modules correspond bijectively to
$G$-stable functorially finite torsion classes as well as $G$-stable two-term silting complexes.

The following result establishes a one-to-one correspondence between $G$-stable support $\tau$-tilting
$\Lambda$-modules and $G$-stable functorially finite torsion classes in $\mod\Lambda$.

\begin{theorem} The bijection of Theorem 2.3 restricts to a bijection:
\begin{center}
$G$-$s$$\tau$-$\tilt\Lambda$ $\longleftrightarrow G$-$f$-$\tors\Lambda$.
\end{center}
\end{theorem}

\begin{proof} Assume that $T$ is $G$-stable. For any $M\in$ $\mod\Lambda$ and $\sigma\in G$,
we have that for any $n\geq 1$, $T^n\rightarrow M$ is surjective if and only if so is $(^{\sigma}T)^n\rightarrow{^{\sigma}\hspace{-2pt}M}$. So we have
$$\opname{Fac}T=\opname{Fac} {^{\sigma}\hspace{-2pt}T}={^{\sigma}\hspace{-2pt}\opname{Fac} T}$$
for any $\sigma\in G$, that is, $\opname{Fac} T$ is $G$-stable.

Conversely, if $\mathcal{T}\in f$-$\tors\Lambda$ is $G$-stable, then $^{\sigma}\mathcal{T}= \mathcal{T}$ for any $\sigma\in G$.
Since $^{\sigma}(-):\mod\Lambda\rightarrow \mod\Lambda$ is an automorphism, we have
$$\Ext^{1}_{\Lambda}(-, ^{\sigma}\hspace{-2pt}\mathcal{T})\cong\Ext^{1}_{\Lambda}(^{\sigma^{-1}}(-), \mathcal{T}).$$
So $$P(\mathcal{T})=P(^{\sigma}\mathcal{T})\cong{^{\sigma}\hspace{-2pt}P(\mathcal{T})}$$
for any $\sigma\in G$, that is, $P(\mathcal{T})$ is $G$-stable.
\end{proof}

Recall that $M\in\mod\Lambda$ is {\it sincere} if every simple $\Lambda$-module appears as a composition factor in $M$.
This is equivalent to the condition that $\Hom_{\Lambda}(P, M)\neq 0$ for any indecomposable projective module $P$.
Also recall that $M\in \mod\Lambda$ is {\it faithful} if its left annihilator
$$\opname{Ann} M:= \{\lambda\in\Lambda\mid \lambda M=0\}=0.$$
A class of left $\Lambda$-modules $\mathcal{T}$ is {\it sincere} if for any indecomposable projective module $P$,
we have $$\Hom_{\Lambda}(P, \mathcal{T}):=\{\Hom_{\Lambda}(P, T)\mid T\in\mathcal{T}\}\neq 0.$$
A class of left $\Lambda$-modules $\mathcal{T}$ is {\it faithful} if
$$\opname{Ann} \mathcal{T}=\bigcap_{T\in\mathcal{T}}\opname{Ann}T=0.$$

The following result shows that support $\tau$-tilting modules can be regarded as a common generalization of $\tau$-tilting modules and tilting modules.

\begin{proposition} {\rm ([AIR, Proposition 2.2])}
\begin{itemize}
\item[(1)] $\tau$-tilting modules are precisely sincere support $\tau$-tilting modules.
\item[(2)] Tilting modules are precisely faithful support $\tau$-tilting modules.
\end{itemize}
\end{proposition}

We denote by $G$-$s\hspace{-2pt}f$-$\tors\Lambda$ (respectively, $G$-$f\hspace{-3pt}f$-$\tors\Lambda$) the set of $G$-stable sincere
(respectively, faithful) functorially finite torsion classes in $\mod\Lambda$. Using Proposition 3.5, we get the following

\begin{theorem} The bijection in Theorem 3.4 restricts to bijections
\begin{center}
$G$-$\tau$-$\tilt\Lambda \longleftrightarrow G$-sf-$\tors\Lambda$ and $G$-$\tilt\Lambda \longleftrightarrow G$-ff-$\tors\Lambda$.
\end{center}
\end{theorem}

\begin{proof} Let $T$ be a $G$-stable support $\tau$-tilting $\Lambda$-module. It follows Proposition 3.5 that
$T$ is a $\tau$-tilting $\Lambda$-module (respectively, tilting $\Lambda$-module) if and only if $T$ is sincere (respectively, faithful).

Claim 1: $T$ is sincere if and only if $\opname{Fac} T$ is sincere.

If $T$ is sincere, it is obvious that $\opname{Fac} T$ is sincere by definition.

Conversely, if $\opname{Fac} T$ is sincere, then for any indecomposable projective module $P$ we have $\Hom_{\Lambda}(P, \opname{Fac} T)\neq 0$,
that is, there exists $M_{P}\in\opname{Fac} T$ such that $\Hom_{\Lambda}(P, M_{P})\neq 0$.
Since $M_{P}\in\opname{Fac} T$,
there exist exact sequences: $$T^n\rightarrow M_{P}\rightarrow 0,\ {\rm and}$$
$$\Hom_{\Lambda}(P, T^n)\rightarrow \Hom_{\Lambda}(P, M_{P})\rightarrow 0,$$ where $n\geq 1$.
So we have $\Hom_{\Lambda}(P, T)\neq 0$ for any indecomposable projective module $P$. Therefore $T$ is sincere.

Claim 2: $T$ is faithful if and only if $\opname{Fac} T$ is faithful.

It suffices to show that $\opname{Ann}T=\opname{Ann}\opname{Fac} T$. It is obvious that $\opname{Ann}\opname{Fac} T\subseteq$ $\opname{Ann}T$
by definition. Conversely, for any $\lambda\in\opname{Ann}T$ and $M\in\opname{Fac} T$, there exists an exact sequence
$T^n\buildrel {f} \over\rightarrow M\rightarrow 0$ with $n\geq 1$ and $\lambda T=0$. Then we have
\begin{center}
$\lambda M=\lambda f(T^n)= f(\lambda T^n)=0$,
\end{center}
that is, $\opname{Ann}T\subseteq$ $\opname{Ann}\opname{Fac} T$.
\end{proof}

We end this subsection with the following result.
\begin{theorem} The bijection of Theorem 2.4 restricts to a bijection:
\begin{center}
$G$-$2$-$\silt\Lambda\longleftrightarrow G$-$s\tau$-$\tilt\Lambda$.
\end{center}
\end{theorem}
\begin{proof}  If  $(T,P)\in G$-$s$$\tau$-$\tilt\Lambda$, then $^{\sigma}\hspace{-1pt}T\cong T$ and $^{\sigma}\hspace{-2pt}P\cong P$ for any
$\sigma\in G$. Let $P^{\bullet}\rightarrow T\rightarrow 0$ be a minimal projective presentation of $T$. Then
$^{\sigma}\hspace{-2pt}P^{\bullet}\rightarrow$ $^{\sigma} \hspace{-1pt}T\rightarrow 0$ is a minimal projective presentation of
$^{\sigma}\hspace{-1pt}T$ since $^{\sigma}(-): \mod\Lambda\rightarrow \mod\Lambda$ is an automorphism. Since $^{\sigma}\hspace{-1pt}T\cong T$
for any $\sigma\in G$, it follows that $^{\sigma}\hspace{-2pt}P^{\bullet}\cong P^{\bullet}$ and $^{\sigma}\hspace{-2pt}(P^{\bullet}\oplus P[1])
\cong P^{\bullet}\oplus P[1]$.

Conversely, if $P^{\bullet}\in G$-$2$-$\silt\Lambda$, then $H^{0}(P^{\bullet}) \cong H^{0}(^{\sigma}\hspace{-2pt}P^{\bullet})=
{^{\sigma}\hspace{-2pt} H^{0}(P^{\bullet}})$ because $\sigma$ commutes with taking cokernel. It follows that $H^{0}(P^{\bullet})$ is $G$-stable.
\end{proof}

\vspace{0.2cm}

{\bf 3.3. Connection of $G$-$s$$\tau$-$\tilt\Lambda$ with $G$-$c$-tilt$\mathcal{C}$}

\vspace{0.2cm}

Let $\mathcal{C}$ be a $k$-linear Hom-finite Krull-Schimidt 2-Calabi-Yau triangulated category. An action of $G$ on $\mathcal{C}$ is a group homomorphism
$\theta: G\rightarrow \Aut\mathcal{C}$ from $G$ to the group of triangulated automorphisms of $\mathcal{C}$, that is, $^{\sigma}(-)=\theta(\sigma):
\mathcal{C}\rightarrow \mathcal{C}$ is a triangle automorphism. For any $X\in \mathcal{C}$ and $\sigma\in G$,
$^{\sigma}\hspace{-2pt}X$ denotes the image of $X$ under $^{\sigma}(-)$. An object $X$ in $\mathcal{C}$ is called
{\it $G$-stable} if $^{\sigma}\hspace{-2pt}X \cong X$ for any $\sigma\in G$. We denote by $G$-$c$-$\tilt\mathcal{C}$ the set of isomorphism
classes of basic $G$-stable cluster-tilting objects.

Throughout this subsection, let $T$ be a $G$-stable cluster-tilting object in $\mathcal{C}$ with a fixed isomorphism $\varphi_{\sigma}:
{^{\sigma}T}\rightarrow T$ such that $\varphi_{\sigma\eta}=\varphi_{\sigma}\varphi_{\eta}$ and $\varphi_{1}={\rm id}_{T}$.
Then $\Lambda:=\End_{\mathcal{C}}(T)^{op}$ admits a $G$-action via $\sigma(\lambda)=\varphi_{\sigma}\circ{^{\sigma}\lambda}
\circ\varphi_{\sigma}^{-1}$ for any $\sigma\in G$ and $\lambda\in\Lambda$.
The following proposition plays an important role in this subsection.

\begin{proposition} The action of $G$ on $\mathcal{C}$ induces that on $\mathcal{C}/[T[1]]$.
\end{proposition}
\begin{proof} It suffices to show that $\add T[1]$ is closed under the $G$-action. By the definition of
triangle functors we have that the shift functor [1] and the functor $^{\sigma}(-)$ commute on objects. So
\begin{center}
$^{\sigma}\hspace{-2pt}(\add T[1])=\add ^{\sigma}\hspace{-2pt}(T[1]) = \add (^{\sigma}\hspace{-2pt}T)[1] = \add T[1]$,
\end{center}
and the assertion follows.
\end{proof}

By Theorem 2.5, there exists an equivalence of categories between $\mathcal{C}/[T[1]]$ and $\mod\Lambda$. So the action of $G$
on $\mathcal{C}/[T[1]]$ induces an action of $G$ on $\mod\Lambda$. On the other hand, the action of $G$ on $\Lambda$ also
induces that on $\mod\Lambda$. The following result shows that these two actions coincide.


\begin{lemma} The functor $\Hom_{\mathcal{C}}(T,-): \mathcal{C}\rightarrow \mod\Lambda$ commutes with $G$-action.
\end{lemma}

\begin{proof}
It suffices to prove that for any $M\in\mathcal{C}$, there exists a $\Lambda$-module isomorphism:
$$^{\sigma}\hspace{-2pt}\Hom_{\mathcal{C}}(T,M) \cong \Hom_{\mathcal{C}}(T, {^{\sigma}\hspace{-2pt}M}).$$
Take $$\Phi: {^{\sigma}\hspace{-2pt}\Hom_{\mathcal{C}}(T,M)} \rightarrow
\Hom_{\mathcal{C}}(T, {^{\sigma}\hspace{-2pt}M}) \ {\rm via}\  g\mapsto {^{\sigma}(g\circ{^{\sigma^{-1}}(\varphi_{\sigma}^{-1})})}=
{^{\sigma}g\circ\varphi_{\sigma}^{-1}}$$ for any $g\in{^{\sigma}\hspace{-2pt}\Hom_{\mathcal{C}}(T,M)}$.
Then $\Phi$ is clearly  an isomorphism as $k$-vector spaces. Because
\begin{align*}
&\ \ \ \ \Phi(\lambda\cdot g)\\ &
=\Phi(\sigma^{-1}(\lambda)\cdot g)\\ &
= \Phi(g\circ{^{\sigma^{-1}}(\varphi_{\sigma}^{-1})}
\circ{^{\sigma^{-1}}\lambda}\circ{^{\sigma^{-1}}\varphi_{\sigma}})\\ &
={^{\sigma}g}\circ{\varphi_{\sigma}^{-1}}\circ\lambda\circ{\varphi_{\sigma}}\circ{\varphi_{\sigma}^{-1}}\\ &
={^{\sigma}g}\circ{\varphi_{\sigma}^{-1}}\circ\lambda \\ &
=\lambda\cdot({^{\sigma}g}\circ{\varphi_{\sigma}^{-1}})\\ &
=\lambda\cdot\Phi(g),
\end{align*}
we have that $\Phi$ is a $\Lambda$-module isomorphism.
\end{proof}

\begin{lemma} If both $X_{1}=X\oplus Y'$ and $X_{2}=X\oplus Y''$ are basic cluster-tilting objects with $Y'$ (respectively, $Y''$) a maximal
direct summand of $X_{1}$ (respectively, $X_{2}$) which belongs to $\add T[1]$ , then $Y'\cong Y''$.
\end{lemma}

\begin{proof} Let $X_{1}$, $X_{2}$ and $T$ be cluster-tilting and $Y', Y''\in\add T[1]$. Then we have
$$\Hom_{\mathcal{C}}(X_{2}, Y'[1])=\Hom_{\mathcal{C}}(X, Y'[1])\oplus\Hom_{\mathcal{C}}(Y'', Y'[1])=0,$$
$$\Hom_{\mathcal{C}}(X_{1}, Y''[1])=\Hom_{\mathcal{C}}(X, Y''[1])\oplus\Hom_{\mathcal{C}}(Y', Y''[1])=0.$$
It is easy to get that $Y'\in \add X_{2}$ and $Y''\in \add X_{1}$ from the definition of cluster-tilting objects.
Since $Y'$ (respectively, $Y''$) is a maximal direct summand of $X_{1}$ (respectively, $X_{2}$) which belongs to $\add T[1]$
by assumption, we have $Y'\cong Y''$.
\end{proof}

\begin{lemma}
For $\sigma\in G$, we have $X\in c$-$\tilt\mathcal{C}$ if and only if $^{\sigma}\hspace{-2pt}X\in c$-$\tilt\mathcal{C}$.
\end{lemma}

\begin{proof}
Let $X\in c$-$\tilt\mathcal{C}$. Then $\add X=\{C\in\mathcal{C}\mid \Hom_{\mathcal{C}}(X,C[1])=0\}$. So we have
\begin{align*}
&\ \ \ \ \add {^{\sigma}X}\\ &
={^{\sigma}\add X}\\ &
=\{{^{\sigma}C}\in\mathcal{C}\mid \Hom_{\mathcal{C}}(X,C[1])=0\}\\ &
=
\{{^{\sigma}C}\in\mathcal{C}\mid \Hom_{\mathcal{C}}({^{\sigma}X},{^{\sigma}C}[1])=0\}\\ &
=\{C\in\mathcal{C}\mid \Hom_{\mathcal{C}}({^{\sigma}X},C[1])=0\},
\end{align*}
and hence $^{\sigma}\hspace{-2pt}X\in c$-$\tilt\mathcal{C}$.
Dually, we get that $^{\sigma}\hspace{-2pt}X\in c$-$\tilt\mathcal{C}$ implies $X\in c$-$\tilt\mathcal{C}$.
\end{proof}

The following observation is useful.

\begin{proposition}
If $X$=$X'\oplus X''\in$ $c$-$\tilt\mathcal{C}$ with $X''$ a maximal direct summand of $X$ which belongs to $\add T[1]$,
then $X$ is $G$-stable if and only if $X'$ is $G$-stable.
\end{proposition}

\begin{proof}
By Proposition 3.8, we have
$$^{\sigma}\hspace{-2pt}\add T[1] = \add {^{\sigma}\hspace{-2pt}(T[1])}=\add(^{\sigma}T)[1] = \add T[1].$$
So $Y \in \add T[1]$ if and only if $^{\sigma}\hspace{-1pt}Y \in \add T[1]$.
Since $X''$ is a maximal direct summand of $X$ which belongs to $\add T[1]$ by assumption,
$^{\sigma}\hspace{-2pt}X''$ is also a maximal direct summand of $^{\sigma}\hspace{-2pt}X\in \add T[1]$.

If $X$ is $G$-stable, then $^{\sigma}\hspace{-2pt}X\cong X$ for any $\sigma\in G$. Since $^{\sigma}\hspace{-2pt}X''$ (respectively, $X''$)
is a maximal direct summand of $^{\sigma}\hspace{-2pt}X$ (respectively, $X$) which belongs to $\add T[1]$, we have $^{\sigma}\hspace{-2pt}X''
\cong X''$ and $^{\sigma}\hspace{-2pt}X' \cong X'$ for any $\sigma\in G$.

By Lemma 3.11 we have $^{\sigma}\hspace{-2pt}X'\oplus$ $^{\sigma}\hspace{-2pt}X''\in$ c-tilt$\mathcal{C}$.
If $X'$ is $G$-stable, then $X'\oplus$ $^{\sigma}\hspace{-2pt}X''$ and $X'\oplus X''$ are basic cluster-tilting objects.
By Lemma 3.10 we have $^{\sigma}\hspace{-2pt}X'' \cong X''$. Thus $X$ is $G$-stable.
\end{proof}
Now we are in a position to prove the following

\begin{theorem}
The bijection of Theorem 2.6 restricts to a bijection:
\begin{center}
$G$-$c$-$\tilt\mathcal{C} \longleftrightarrow G$-$s$$\tau$-$\tilt\Lambda$.
\end{center}
\end{theorem}

\begin{proof}  By Lemma 3.9, we have $^{\sigma}\hspace{-2pt}\Hom_{\mathcal{C}}(T,M) \cong \Hom_{\mathcal{C}}$($T$, $^{\sigma}\hspace{-2pt}M$) for any $M\in \mathcal{C}/[T[1]]$.

If $X\in G$-c-$\tilt\mathcal{C}$, then $X'$ is $G$-stable by Proposition 3.12. For any $\sigma\in G$, we have
$$^{\sigma}\overline{X'}=
{^{\sigma}\hspace{-2pt}\Hom_{\mathcal{C}}(T,X')} \cong \Hom_{\mathcal{C}}(T, {^{\sigma}\hspace{-2pt}X'})\cong \Hom_{\mathcal{C}}(T,X')
= \overline{X'}.$$ Thus $\overline{X'}$ is $G$-stable. Moreover, $\widetilde{X}\in G$-s$\tau$-$\tilt\Lambda$ is $G$-stable.

Conversely, if $\overline{X'}$ is $G$-stable, then we have $\Hom_{\mathcal{C}}$($T$, $^{\sigma}\hspace{-2pt}X'$) $\cong \Hom_{\mathcal{C}}(T,X')$ as before.
Then by Theorem 2.5, we have $^{\sigma}\hspace{-2pt}X' \cong X'$ for any $\sigma\in G$, that is, $X'$ is $G$-stable. Thus $X\in G$-$c$-$\tilt\mathcal{C}$
by Proposition 3.12.
\end{proof}

In the following we establish a bijection between $G$-stable cluster-tilting objects in $\mathcal{C}$ and $G$-stable
two-term silting complexes in $\opname{K}^{b}(\proj\Lambda)$.

\begin{lemma}
Let $X$ be a basic object of $\mathcal{C}$ and take a triangle
$$T_{1}\buildrel {g} \over\rightarrow T_{0}\buildrel {f} \over\rightarrow X\rightarrow T_{1}[1]$$
with $T_{1}, T_{0}\in \add T$ and $f$ a minimal right $\add T$-approximation. Then the following statements are equivalent.
\begin{itemize}
\item[(1)] $X$ is $G$-stable in $\mathcal{C}$.
\item[(2)] $\overline{T_{1}}\buildrel {\overline{g}} \over\rightarrow\overline{T_{0}}$ is $G$-stable in $\opname K^{b}(\proj\Lambda)$.
\end{itemize}
\end{lemma}

\begin{proof}
Since $^{\sigma}(-): \mathcal{C} \rightarrow \mathcal{C}$ is a triangulated equivalence. So we have the following diagram:
$$\xymatrix{
& T_{1}\ar[r]^{g}\ar[d] & T_{0}\ar[r]^{f}\ar[d] & X\ar[r]\ar[d] & T_{1}[1]\ar[d]&\\
& ^{\sigma}\hspace{-2pt}T_{1}\ar[r]^{^{\sigma}\hspace{-2pt}g} & ^{\sigma}\hspace{-2pt}T_{0}\ar[r]^{^{\sigma}\hspace{-2pt}f}
& ^{\sigma}\hspace{-2pt}X\ar[r] & (^{\sigma}\hspace{-2pt}T_{1})[1].&}$$
By the proof of Theorem 3.13, we have that $T_{1}\buildrel {g} \over\rightarrow T_{0}$ is $G$-stable if and only if
$\overline{T_{1}}\buildrel {\overline{g}} \over\rightarrow\overline{T_{0}}$ is $G$-stable.

\vspace{0.2cm}

$(2)\Rightarrow (1)$ We already have $T_{1}\buildrel {g} \over\rightarrow T_{0}$ is $G$-stable, it follows that $X$ is $G$-stable.

$(1)\Rightarrow (2)$ If $X$ is $G$-stable, then $^{\sigma}\hspace{-2pt}X \cong X$. Since $f$ is a minimal right $\add T$-approximation by assumption,
$^{\sigma}\hspace{-2pt}f$ is a minimal right $^{\sigma}\hspace{-2pt}\add T$-approximation and hence a minimal right $\add T$-approximation. Thus
$^{\sigma}\hspace{-1pt}T_{0} \cong T_{0}$ and $^{\sigma}\hspace{-1pt}T_{1} \cong T_{1}$. Therefore $T_{1}\buildrel {g} \over\rightarrow T_{0}$
is $G$-stable and the assertion follows.
\end{proof}

Immediately, we get the following

\begin{theorem}
The bijection of Theorem 2.7 restricts to a bijection:
\begin{center}
$G$-$c$-$\tilt\mathcal{C} \longleftrightarrow G$-$2$-$\silt\Lambda$.
\end{center}
\end{theorem}

\begin{proof}
It follows from Theorem 2.7 and Lemma 3.14.
\end{proof}

\section{relationship between stable support $\tau$-tilting modules over $\Lambda$ and $\Lambda G$}

\vspace{0.2cm}

Throughout this section, $\Lambda$ is an algebra and $G$ is a finite group acting on $\Lambda$ such that
$|G|$ is invertible in $\Lambda$. We denote by $\mathbb{X}$ the group of characters on $G$, that is, the group homomorphisms $\chi: G\rightarrow k^{\ast}=k\backslash\{0\}$.
Then $X$ acts on $\Lambda G$ via $\chi(\lambda g)=\chi(g)\lambda g$.
We prove that there exists an injection from $G$-stable support $\tau$-tilting $\Lambda$-modules to $\mathbb{X}$-stable support $\tau$-tilting $\Lambda G$-modules.
In the case when $G$ is a solvable group, the injection turns out to be a bijection.

We begin with the following easy observation.

\begin{lemma}
If $T\in \mod\Lambda$ is $G$-stable, then $FT$ is $\mathbb{X}$-stable.
\end{lemma}

\begin{proof}
We only need to prove $^{\chi}FT \cong FT$ for any $\chi\in X$. Note that $^{\chi}FT$ is a $\Lambda G$-module whose underlying set and the
additive structure is the same as $FT$, in which $(\lambda',g')\circ((\lambda,g)\otimes t)$ is defined to be $\chi(g')(\lambda',g')(\lambda,g)\otimes t$.
Define $\theta:$ $^{\chi}FT\rightarrow FT$ via $(\lambda,g)\otimes t\mapsto \chi^{-1}(g)(\lambda,g)\otimes t$. Clearly it is a bijection. Because
\begin{eqnarray*}
& &\theta((\lambda',g')\circ(\lambda,g)\otimes t)\\
&=&\theta(\chi(g')(\lambda',g')(\lambda,g)\otimes t)\\
&=&\chi(g')\theta((\lambda',g')(\lambda,g)\otimes t)\\ &=&\chi(g')\theta((\lambda'g'(\lambda),g'g\otimes t)\\
&=&\chi(g')\chi^{-1}(g'g)(\lambda',g')(\lambda,g)\otimes t\\ &=&\chi^{-1}(g)(\lambda',g')(\lambda,g)\otimes t\\
&=&(\lambda',g')\theta((\lambda,g)\otimes t),
\end{eqnarray*}
we have that $\theta$ is a $\Lambda G$-homomorphism, and hence an isomorphism.
\end{proof}

The first main result in this section is the following.

\begin{theorem} The functor $$F=\Lambda G\otimes_{\Lambda}-: \mod\Lambda \rightarrow \mod\Lambda G$$ via $T \mapsto FT$ induces the following injections:
\begin{itemize}
\item[(1)] from the set of isomorphism classes of $G$-stable $\tau$-rigid $\Lambda$-modules to the set of isomorphism classes of
$\mathbb{X}$-stable $\tau$-rigid $\Lambda G$-modules.
\item[(2)] from the set of isomorphism classes of $G$-stable $\tau$-rigid pair in $\mod\Lambda$ to
the set of isomorphism classes of $\mathbb{X}$-stable $\tau$-rigid pair in $\mod\Lambda G$.
\item[(3)] $G$-s$\tau$-$\tilt\Lambda\rightarrow \mathbb{X}$-$s\tau$-$\tilt\Lambda G$.
\end{itemize}
\end{theorem}

\begin{proof}  We claim that the functor $F$ restricting to the set of isomorphism classes of basic $G$-stable $\Lambda$-modules is an injection.
If both $T_{1}$ and $T_{2}$ are $G$-stable $\Lambda$-modules and $FT_{1} \cong FT_{2}$, then $HFT_{1} \cong HFT_{2}$, that is, $\bigoplus_{\sigma\in G}{^{\sigma}T_{1}}
\cong \bigoplus_{\sigma\in G}{^{\sigma}T_{2}}$. Since $T_{1}$ and $T_{2}$ are $G$-stable, we have $T^n_{1} \cong T^n_{2}$ with $n=|G|$.
Thus $T_{1} \cong T_{2}$, and the claim follows.

(1) By definition, $T$ is $\tau$-rigid if and only if $\Hom_{\Lambda}(T, \tau T)=0$. By the proof of [RR, Lemma 4.2], we have that $F$ commutes with $\tau$. So we have
\begin{align*}
\Hom_{\Lambda G}(FT, \tau FT) &\cong \Hom_{\Lambda G}(FT, F\tau T) \cong \Hom_{\Lambda}(T, HF\tau T)\\ &\cong
 \Hom_{\Lambda}(T, \bigoplus_{\sigma\in G} {^{\sigma}(\tau T)} ) \cong \bigoplus_{\sigma\in G}\Hom_{\Lambda}(T, {^{\sigma}(\tau T))}\\
&\cong\bigoplus_{\sigma\in G}\Hom_{\Lambda}(^{\sigma^{-1}}T, \tau T) \cong\ \bigoplus_{\sigma\in G}\Hom_{\Lambda}(T, \tau T).
\end{align*}
Since $T$ is $G$-stable $\tau$-rigid in $\mod\Lambda$, we have $FT$ is $\tau$-rigid in $\mod\Lambda G$. Now the assertion follows from Lemma 4.1.

(2) Note that $(T,P)$ is a $G$-stable $\tau$-rigid pair if and only if $T$ is $G$-stable $\tau$-rigid, $P$ is $G$-stable projective and $\Hom_{\Lambda}(P,T)=0$. It follows from the injection in (1)
that $FT$ is $\tau$-rigid in $\mod\Lambda G$. By Lemma 2.8, $F$ preserves projective modules and $FP$ is a projective module in $\mod\Lambda G$. We have
\begin{align*}
\Hom_{\Lambda G}(FP,FT) &\cong \Hom_{\Lambda}(P,HFT) \cong \Hom_{\Lambda}(P,\bigoplus_{\sigma\in G} {^{\sigma}T} )\\
&\cong \bigoplus_{\sigma\in G}\Hom_{\Lambda}(P, {^{\sigma}T)} \cong \bigoplus_{\sigma\in G}\Hom_{\Lambda}(P, T).
\end{align*}
Thus $\Hom_{\Lambda G}(FP,FT)$=0, and therefore $(FT,FP)$ is a $\mathbb{X}$-stable $\tau$-rigid pair in $\mod\Lambda G$ by Lemma 4.1.

\vspace{0.2cm}

(3) By Propositions 2.1 and 2.2, $T\in s\tau$-$\tilt\Lambda$ if and only if $T$ is $\tau$-rigid and there exists an exact sequence
$$\Lambda\buildrel {f} \over\rightarrow T'\buildrel {g} \over\rightarrow T''\rightarrow 0$$
in $\mod \Lambda$ with $T', T''\in \add T$ and $f$ a left $\add T$-approximation of $\Lambda$. It follows from the injection in (1) that
$FT$ is a $\tau$-rigid $\Lambda G$-module and there exists an exact sequence
$$F\Lambda(\cong \Lambda G)\buildrel {Ff} \over\rightarrow FT'\buildrel {Fg} \over\rightarrow FT''\rightarrow 0$$
in $\mod\Lambda G$ with $FT',FT''\in \add FT$.
Then by Proposition 2.2, we only have to prove that $Ff$ is a left $\add FT$-approximation of $\Lambda G$, that is, $\Hom_{\Lambda G}(Ff,M)$
is surjective for any $M\in\add FT$. First we prove that $\Hom_{\Lambda G}(Ff,FT)$ is surjective. Consider the following commutative diagram:
$$\xymatrix{
& \Hom_{\Lambda G}(FT',FT)\ar[r]^{\Hom_{\Lambda G}(Ff,FT)}\ar[d]^{\cong} & \Hom_{\Lambda G}(\Lambda G,FT)\ar[d]^{\cong} &\\
& \Hom_{\Lambda}(T', HFT)\ar[r]^{\Hom_{\Lambda}(f,HFT)}\ar[d]^{\cong} & \Hom_{\Lambda}(\Lambda, HFT)\ar[d]^{\cong} &\\
& \Hom_{\Lambda}(T', \bigoplus_{\sigma\in G}{^{\sigma}T})\ar[r]^{\Hom_{\Lambda}(f,\bigoplus_{\sigma\in G}{^{\sigma}T})}
\ar[d]^{\cong} & \Hom_{\Lambda}(\Lambda, \bigoplus_{\sigma\in G}{^{\sigma}T})\ar[d]^{\cong} &\\
& \Hom_{\Lambda}(T', T^n)\ar[r]^{\Hom_{\Lambda}(f,T^n)} & \Hom_{\Lambda}(\Lambda, T^n), &
}$$
where $n=|G|$. The last row is surjective since $f$ is left $\add T$-approximation of $\Lambda$. So the first row is also surjective.

Now let $M\in\add FT$ and $g\in\Hom_{\Lambda G}(\Lambda G,M)$. Then there exist $m\geq 1$ and $N\in \mod \Lambda G$
such that $M\oplus N\cong (FT)^{m}$. So we have a split exact sequence
$$\xymatrix{
0\ar[r]& M\ar[r]^{i} & (FT)^{m}\ar[r]\ar@<1ex>[l]^{p} & N\ar[r] & 0&
}$$ in $\mod \Lambda G$ with $pi=1_{M}$. By the above argument,
there exists $h\in\Hom_{\Lambda G}(FT',(FT)^{m})$ such that $hFf=ig$. So we have
$$g=pig=phFf=\Hom_{\Lambda G}(Ff,M)(ph),$$ and hence $\Hom_{\Lambda G}(Ff,M)$
is surjective.
\end{proof}

As an application of Theorem 4.2, we get the following result which extends [DLS, Proposition 3.1.1].

\begin{corollary}
If $T$ is a $G$-stable (basic) tilting $\Lambda$-module, then $FT$ is a $\mathbb{X}$-stable tilting $\Lambda G$-module.
\end{corollary}

\begin{proof}
Let $T$ be a $G$-stable (basic) tilting $\Lambda$-module. Then by Proposition 3.5(2), $T$ is a $G$-stable faithful support $\tau$-tilting module.
So $\Lambda$ is cogenerated by $T$ and there exists an injection $0\rightarrow \Lambda \rightarrow T^n$ in $\mod \Lambda$. Since $F$ is exact,
we get an injection $0\rightarrow \Lambda G \rightarrow (FT)^n$ in $\mod \Lambda G$. So $FT$ is a $\mathbb{X}$-stable faithful support $\tau$-tilting
$\Lambda G$-module by Theorem 4.2(3), and hence it is a tilting $\Lambda G$-module by Proposition 3.5(2) again.
\end{proof}

The following observation is standard.

\begin{proposition}
Any $\Lambda G$-module is a $G$-stable $\Lambda$-module.
\end{proposition}
\begin{proof} Let $Y$ be a $\Lambda G$-module. For any $g\in G$ and $y\in Y$, we define a map
$$f_{g}:{^{g}Y}\rightarrow Y$$ by $f_{g}(y)=gy$.
Then for any $a\in\Lambda$, we have
$$f_{g}(ay)=g(ay)=g(g^{-1}(a)y)=ag(y)=af_{g}(y).$$ So $f_{g}$ is a $\Lambda$-module homomorphism. We also have that $f_{g}$ is an isomorphism with the inverse
$f_{g^{-1}}: Y\rightarrow$ $^{g}Y$ such that $f_{g^{-1}}(y)=g^{-1}y$ for any $y\in Y$.
\end{proof}

As an immediate consequence of Proposition 4.4, we have the following

\begin{corollary}
For any basic $G$-stable $\Lambda$-module $T$, we have
\begin{itemize}
\item[(1)] If $T$ is $\tau$-rigid in $\mod\Lambda$, then $HFT$ is $G$-stable $\tau$-rigid in $\mod\Lambda$.
\item[(2)] If $T$ is support $\tau$-tilting in $\mod\Lambda$, then  $HFT$ is $G$-stable support $\tau$-tilting in $\mod\Lambda$.
\end{itemize}
\end{corollary}

\begin{proof}
Note that $HFT \cong \bigoplus_{\sigma\in G}$ $^{\sigma}T$ $\cong$ $T^n$ with $n=|G|$. So both assertions follow from Proposition 4.4.
\end{proof}

We have proved in Theorem 4.2(3) that $F$ induces an injection from $G$-$s\tau$-$\tilt\Lambda$ to $\mathbb{X}$-$s\tau$-$\tilt\Lambda G$. It is natural to ask the following question.

\begin{question*}
When is this injection a bijection?
\end{question*}

In the following, we give a partial answer to this question.

It follows from [RR, Corollary 5.2] that $(\Lambda G)\mathbb{X}$ is Morita equivalent to $\Lambda G^{(1)}$, where $G^{(1)}$ is the commutator subgroup of $G$. Let $G$ be solvable, and let
$$G\rhd G^{(1)}\rhd G^{(2)}\rhd G^{(3)}\rhd\cdots \rhd G^{(m)}=\{1\}$$ be its derived series, that is, every subgroup is
the commutator subgroup of the preceding one. Denote by $\mathbb{X}^{(i)}$ the character group of $G^{(i)}$. By [RR, Proposition 5.4],
we can get from $\Lambda G$ to $\Lambda$ by using a finite number of skew group algebra constructions, combined with Morita equivalences.
To be more precise, there exists a chain of skew group algebras
$$\Lambda\buildrel {G} \over\rightarrow \Lambda G\buildrel {\mathbb{X}} \over\rightarrow \Lambda G^{(1)}\buildrel {\mathbb{X}^{(1)}}
\over\rightarrow \Lambda G^{(2)}\rightarrow\cdots \buildrel {\mathbb{X}^{(m-1)}} \over\rightarrow\Lambda G^{(m)}\buildrel {Morita}
\over \simeq \Lambda,$$ where each algebra $\Lambda G^{(i)}$ is the skew group algebra of the preceding algebra $\Lambda G^{(i-1)}$ under the action of the
group $\mathbb{X}^{i-1}$. Then we have the induced functors $$\mod\Lambda\buildrel {F} \over\rightarrow \mod\Lambda G\buildrel {F^{(1)}} \over\rightarrow
\mod\Lambda G^{(1)}\rightarrow\cdots \buildrel {F^{(m)}} \over\rightarrow \mod\Lambda.$$
Under the above assumption, we give the following

\begin{theorem}
If $G$ a solvable group, then the functor $F: \mod\Lambda\rightarrow \mod\Lambda G$ induces a bijection:
\begin{center}
$G$-$s\tau$-$\tilt\Lambda\rightarrow \mathbb{X}$-$s\tau$-$\tilt\Lambda G$.
\end{center}
\end{theorem}

\begin{proof} By Theorem 4.2, the functor $F$ induces an injection:
\begin{center}
$G$-$s\tau$-$\tilt\Lambda\rightarrow \mathbb{X}$-$s\tau$-$\tilt\Lambda G$.
\end{center}
Applying Lemma 4.1 and Theorem 4.2, it is easy to see that the functors $F^{(i)}$ induce injections:
\begin{center}
$\mathbb{X}^{(i-1)}$-$s\tau$-$\tilt\Lambda G^{(i-1)}\rightarrow \mathbb{X}^{(i)}$-$s\tau$-$\tilt\Lambda G^{(i)}$.
\end{center}
Then we have the following chain of injections:
\begin{center}
$\quad\quad\quad\quad\quad G$-$s\tau$-$\tilt\Lambda\buildrel {F} \over\rightarrow \mathbb{X}$-$s\tau$-$\tilt\Lambda G\buildrel {F^{(1)}} \over\rightarrow \mathbb{X}^{(1)}$-$s\tau\text{-}\tilt\Lambda G^{(1)}$
\end{center}
\begin{center}
$\quad\quad\quad\quad\quad\quad\buildrel {F^{(2)}} \over\rightarrow\cdots \buildrel {F^{(m)}} \over\rightarrow G$-$s\tau$-$\tilt\Lambda G^{(m)} \cong G$-$s\tau\text{-}\tilt\Lambda$.
\end{center}
The composition $F^{(m)}\cdots F^{(1)}F$ is a bijection. So $F$ and all $F^{(i)}$ are bijections.
\end{proof}

Let $G$ be an abelian group. It is well known that $\mathbb{X}$ is isomorphic to $G$. So $\Lambda G$ admits an action by $G$; and moreover,
by [RR] the skew group algebra $(\Lambda G)G$ is Morita equivalent to $\Lambda$. Now the following is an immediate consequence of Theorem 4.6.

\begin{corollary}
If $G$ is an abelian group, then $F$ induces a bijection:
\begin{center}
$G$-$s\tau$-$\tilt\Lambda\rightarrow G$-$s\tau$-$\tilt\Lambda G$.
\end{center}
\end{corollary}


Finally, we illustrate Theorem 4.6 with the following example.

\begin{example}
Let $\Lambda$ be the path algebra of the quiver $Q$ (see below), and let $G=\mathbb{Z}/2\mathbb{Z}$ act on $\Lambda$ by switching $2$ and $2'$,
$\alpha$ and $\beta$ and fixing the vertex $1$.
Then the following $Q'$ is the quiver of $\Lambda G$.
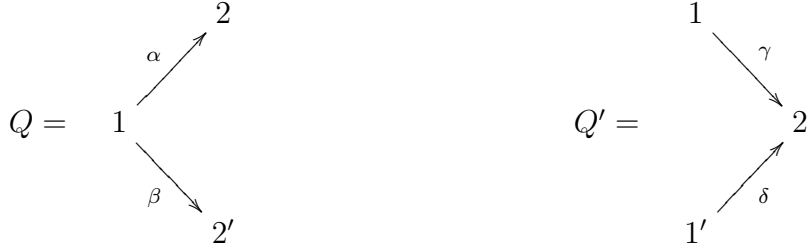
$$\xymatrix{
&&&2 &&&&&1\ar[dr]^{\gamma} &&&\\
&Q=\hspace{-15pt}&1\ar[ur]^{\alpha}\ar[dr]_{\beta}&& &&&Q'=\hspace{-15pt}&&2&\\
&&&2'&&&&&1'\ar[ur]_{\delta}&&&
}$$
The Auslander-Reiten quivers of $\mod\Lambda$ and $\mod\Lambda G$ are the following, where each module is represented by its radical filtration.
$$
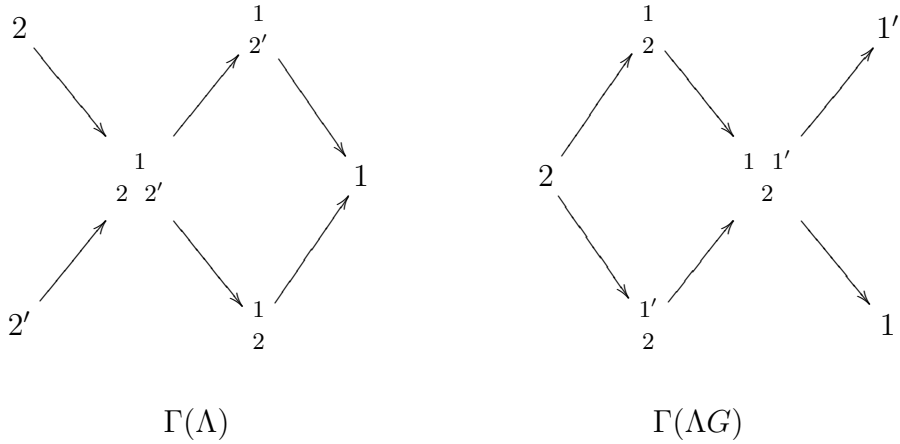
\xymatrix{
&2\ar[dr]&& \scriptsize{\begin{matrix} 1\\2'\\ \end{matrix}}\ar[dr] &&&& \scriptsize{\begin{matrix} 1\\2\\ \end{matrix}}\ar[dr]&& 1'&\\
&& \scriptsize{\begin{matrix} 1\\2\ \ 2'\\ \end{matrix}}\ar[ur]\ar[dr]&&1&&2\ar[ur]\ar[dr] && \scriptsize{\begin{matrix} 1\ \ 1'\\2\\ \end{matrix}}\ar[ur]\ar[dr] &&\\
&2'\ar[ur]&& \scriptsize{\begin{matrix} 1\\2\\ \end{matrix}}\ar[ur] &&&& \scriptsize{\begin{matrix} 1'\\2\\ \end{matrix}}\ar[ur]&& 1&\\
}$$
$$\xymatrix{
\Gamma(\Lambda)&&&&&\Gamma(\Lambda G)
}$$
We denote by $\opname{ind} \Lambda$ the set of isomorphism classes of indecomposable $\Lambda$-modules and by $\opname{ind}\Lambda G$
the set of isomorphism classes of indecomposable $\Lambda G$-modules.

Then we describe the map induced by $F$ between $\opname{ind} \Lambda$ and $\opname{ind} \Lambda G$. Observe that the correspondences from the
the Auslander-Reiten quiver of $\Lambda$ to that of $\Lambda G$:

\vspace{0.5cm}

$\quad\quad\quad\quad\quad\quad\quad\quad\quad\quad\quad F$: $\opname{ind}\Lambda\rightarrow\opname{ind}\Lambda G$

$\quad\quad\quad\quad\quad\quad\quad\quad\quad\quad\quad\quad \ \  2,2' \ \longmapsto$ 2

$\quad\quad\quad\quad\quad\quad\quad\quad\quad\quad\quad\quad\quad$
${\scriptsize{\begin{matrix} 1\\2\ \ 2'\\ \end{matrix}}} \ \longmapsto \scriptsize{\begin{matrix} 1\\2\\ \end{matrix}}\oplus\scriptsize{\begin{matrix} 1\\2'\\ \end{matrix}}$

\vspace{0.2cm}

$\quad\quad\quad\quad\quad\quad\quad\quad\quad\quad\quad\quad$
${\scriptsize{\begin{matrix} 1\\2'\\ \end{matrix}}}$, ${\scriptsize{\begin{matrix} 1\\2\\ \end{matrix}}} \ \ \longmapsto \scriptsize{\begin{matrix} 1\ \ 1'\\2\\ \end{matrix}}$

\vspace{0.2cm}

$\quad\quad\quad\quad\quad\quad\quad\quad\quad\quad\quad\quad\quad\quad 1\ \longmapsto 1\oplus 1'$.

\vspace{0.5cm}

Recall from [AIR] the definition of the support $\tau$-tilting quiver $Q$($s\tau$-tilt$\Lambda$) of $\Lambda$ as follows.
\begin{itemize}
\item[(1)] The set of vertices is $s$$\tau$-tilt$\Lambda$.
\item[(2)] We draw an arrow from $T$ to $U$ if $U$ is a left mutation of $T$ ([AIR, Theorem 2.30]).
\end{itemize}
One can calculate the left mutation of support $\tau$-tilting $\Lambda$-modules by exchanging sequences that are constructed from left approximations.
Therefore we can draw the support $\tau$-tilting quiver of an algebra by its Auslander-Reiten quiver.
Now we draw $Q(s\tau$-$\tilt\Lambda)$ and $Q(s\tau$-$\tilt\Lambda G)$ as follows.

$Q(s\tau$-$\tilt\Lambda)$:
$$\xymatrix{
&&&2\ar[drr] &&&\\
&&\color{red}{2\oplus 2'}\ar[ur]\ar[r] & 2'\ar[rr] & &\color{blue}{0} &\\
&\color{green}{\scriptsize{\begin{matrix} 1\\2\ \ 2'\\ \end{matrix}}\oplus 2'\oplus 2}\ar[ur]\ar[r]\ar[dr] &
\scriptsize{\begin{matrix} 1\\2\ \ 2'\\ \end{matrix}}\oplus\scriptsize{\begin{matrix} 1\\2'\\
\end{matrix}}\oplus 2'\ar[r]\ar[dr] & \scriptsize{\begin{matrix} 1\\2'\\
\end{matrix}}\oplus 2'\ar[u]\ar[r] & 1\oplus\scriptsize{\begin{matrix} 1\\2'\\ \end{matrix}}\ar[dr] &&\\
&&\scriptsize{\begin{matrix} 1\\2\ \ 2'\\ \end{matrix}}\oplus2\oplus\scriptsize{\begin{matrix} 1\\2\\
\end{matrix}}\ar[r]\ar[dr] & \color{orange}{\scriptsize{\begin{matrix} 1\\2\ \ 2'\\
\end{matrix}}\oplus\scriptsize{\begin{matrix} 1\\2'\\ \end{matrix}}\oplus\scriptsize{\begin{matrix} 1\\2\\
\end{matrix}}}\ar[r] & \color{brown}{1\oplus\scriptsize{\begin{matrix} 1\\2'\\
\end{matrix}}\oplus\scriptsize{\begin{matrix} 1\\2\\ \end{matrix}}}\ar[u]\ar[d]& \color{purple}{1}\ar[uu] &\\
&&& 2\bigoplus\scriptsize{\begin{matrix} 1\\2\\
\end{matrix}}\ar[r]\ar@/^3pc/[uuuu] & 1\bigoplus\scriptsize{\begin{matrix} 1\\2\\ \end{matrix}}\ar[ur] &&\\
}$$

$Q(s\tau$-$\tilt\Lambda G)$:
$$\xymatrix{
&& \scriptsize{\begin{matrix} 1'\\2\\ \end{matrix}}\oplus2\ar[r]\ar[dr] & \color{red}{2}\ar[drrr] &&&&\\
& \color{green}{\scriptsize{\begin{matrix} 1\\2\\ \end{matrix}}\oplus\scriptsize{\begin{matrix} 1'\\2\\
\end{matrix}}\oplus2}\ar[ur]\ar[r]\ar[dr] & \scriptsize{\begin{matrix} 1\\2\\
\end{matrix}}\oplus2\ar[ur]\ar@/_1.5pc/[ddrr] &1'\oplus \scriptsize{\begin{matrix} 1'\\2\\ \end{matrix}}\ar[r] & 1'\ar[rr] && \color{blue}{0}&\\
&& \color{orange}{\scriptsize{\begin{matrix} 1\\2\\ \end{matrix}}\oplus\scriptsize{\begin{matrix} 1'\\2\\
\end{matrix}}\oplus\scriptsize{\begin{matrix} 1\ \ 1'\\2\\
\end{matrix}}}\ar[r]\ar[dr] & 1'\oplus\scriptsize{\begin{matrix} 1'\\2\\ \end{matrix}}\oplus\scriptsize{\begin{matrix} 1\ \ 1'\\2\\
\end{matrix}}\ar[u]\ar[r] & \color{brown}{1'\oplus1\oplus\scriptsize{\begin{matrix} 1\ \ 1'\\2\\ \end{matrix}}}\ar[r] & \color{purple}{1'}\oplus1\ar[d]\ar[ul]&&\\
&&&\scriptsize{\begin{matrix} 1\\2\\ \end{matrix}}\oplus1\oplus\scriptsize{\begin{matrix} 1\ \ 1'\\2\\
\end{matrix}}\ar[ur]\ar[r] & \scriptsize{\begin{matrix} 1\\2\\ \end{matrix}}\oplus1\ar[r] & 1\ar[uur] &&\\
}$$
The colored support $\tau$-tilting modules in the graph are all the basic $G$-stable support $\tau$-tilting modules in
$\mod\Lambda$ and $\mod\Lambda G$ respectively. Moreover, the bijection in Theorem 4.6 takes a $G$-stable support $\tau$-tilting module
in $Q(s\tau$-$\tilt\Lambda)$ to that in $Q(s\tau$-$\tilt\Lambda G$) in the same color.
The $G$-stable support $\tau$-tilting modules in green, orange and brown are $G$-stable tilting.
\end{example}

\vspace{0.5cm}

{\bf Acknowledgement}. This research was partially supported by NSFC (Grant No. 11571164) and a Project Funded by
the Priority Academic Program Development of Jiangsu Higher Education Institutions. The authors would like to thank
Dong Yang and Yuefei Zheng for their helpful discussions.

\end{document}